\documentclass[11pt,a4paper]{article}

\usepackage{amsmath,amsthm,mathrsfs,amsfonts}
\usepackage{hyperref}
\usepackage{amsfonts,todonotes}
\usepackage{amssymb}
\usepackage{bbm}
\usepackage{xcolor}
\usepackage[a4paper,text={16.5cm,25.2cm},centering]{geometry}
\usepackage[compact,small]{titlesec}
\usepackage{soul}
\usepackage{graphicx}
\usepackage{siunitx}

\usepackage{multirow}

\usepackage{parskip}
\newtheorem{theorem}{Theorem}
\newtheorem{lemma}{Lemma}
\newtheorem{definition}{Definition}
\newtheorem{proposition}{Proposition}
\newtheorem{corollary}{Corollary}

\newtheorem{remark}{Remark}
\usepackage{booktabs}
\usepackage[shortlabels]{enumitem}
\usepackage{tikz,pgfplots}
\usepackage{natbib}
\setcitestyle{square,numbers}

\newcommand{\indicator}{\mathbf{1}}
\newcommand{\F}{\mathcal{F}}

\newcommand{\R}{\mathbb{cR}}
\newcommand{\M}{\mathcal{M}}
\renewcommand\P{\operatorname{\mathbf{P}}}
\newcommand\E{\operatorname{\mathbf{E}}}
\newcommand\Mx{{\operatorname{\mathbf{M}}_x}}

\pgfplotsset{compat=1.18}

\pgfmathsetmacro{\minimo}{-2} 
\pgfmathsetmacro{\maximo}{2.5} 
\pgfmathsetmacro{\k}{1}
\pgfmathsetmacro{\K}{2}
\pgfmathsetmacro{\xi}{-0.2107856}
\pgfmathsetmacro{\xs}{1.195415}
\pgfmathsetmacro{\nnn}{-0.0003347508} 
\pgfmathsetmacro{\nn}{-0.0094730103} 
\pgfmathsetmacro{\pp}{0.0592888971} 
\pgfmathsetmacro{\ppp}{0.0011816918} 
\pgfmathsetmacro{\b}{-12*\pp/7} 
\pgfmathsetmacro{\bb}{-12*\ppp/5} 
\pgfmathsetmacro{\bbb}{15/11} 
\pgfmathsetmacro{\bbbb}{-\k} 
\pgfmathsetmacro{\c}{-12*\nn/7} 
\pgfmathsetmacro{\cc}{-12*\nnn/5} 
\pgfmathsetmacro{\ccc}{3/5} 
\pgfmathsetmacro{\cccc}{-\K} 
\pgfmathdeclarefunction{g1}{1}{\pgfmathparse{exp(#1)-\K}}
\pgfmathdeclarefunction{g2}{1}{\pgfmathparse{exp(#1)-\k}}
\pgfmathdeclarefunction{v}{1}{\pgfmathparse{\nnn*exp(-4*(#1-\xs))+\nn*exp(-2*(#1-\xs))+\pp*exp(2*(#1-\xi))+\ppp*exp(4*(#1-\xi))}}
\pgfmathdeclarefunction{qi}{1}{\pgfmathparse{\b*exp(2*(#1-\xi))+\bb*exp(4*(#1-\xi))+\bbb*exp(#1)+\bbbb}}
\pgfmathdeclarefunction{qs}{1}{\pgfmathparse{\c*exp(-2*(#1-\xs))+\cc*exp(-4*(#1-\xs))+\ccc*exp(#1)+\cccc}}
\pgfmathsetmacro{\xval}{1.195415}
\pgfmathparse{g1(\xval)}

\pgfmathsetmacro{\xval}{\xi}
\pgfmathparse{g2(\xval)}

\begin{document}

\title{Dynkin Games for L\'evy Processes}

	\author{ Laura Aspirot \thanks{laura.aspirot@fcea.edu.uy} \hspace{0.1cm} \, \hspace{.2cm}   
	Ernesto Mordecki \thanks{mordecki@cmat.edu.uy}  \hspace{.2cm}    Andres Sosa \thanks{andres.sosa@fcea.edu.uy}
	}
	\maketitle

\begin{abstract}
We obtain a verification theorem for solving a Dynkin game driven by a L\'evy process. 
The result requires finding two averaging functions that, 
composed respectively with the supremum and the infimum of the process, 
summed, and taking the expectation, provide the value function of the game. 
The optimal stopping rules are the respective hitting times of the support sets of the averaging functions. 
We also provide some hints to understand in which type of problems the
verification result can be used.
The proof relies on fluctuation identities of the underlying L\'evy process. 
We illustrate our result with examples, three of them with affine payoffs for different processes, and a fourth consisting of a mathematical finance application.
The smooth pasting property of the solutions is not always present.
\end{abstract}

\section{Introduction}

By a Dynkin game we understand an optimal stopping game involving two players, called the \emph{max} and \emph{min} players. 
The max player chooses the moment to stop a stochastic process, and receives the value of the stopped process from the min player;
the min player chooses the moment to stop another stochastic process, and has to pay the value of this second stopped process to the max player.
The game ends once the first of these two moments has been chosen.
In the present paper, these moments are modeled by stopping times under a common filtration, 
and the payoffs processes are compositions with different functions 
$G_1$ and $G_2$ of a single stochastic processes 
$X=\{X_t\colon t\geq 0\}$,
an underlying L\'evy process. 
As a two player game, when a minimax theorem holds, we say that the game has a value, 
and if there exists a pair of strategies that realizes this value, 
this pair of stopping times is called a Nash equilibrium.



A variant of this game was first formulated by Dynkin in 1969 
(see \cite{dynkin1969}), 
as an extension of the classical optimal stopping problem for sequences of random variables, 
proving the existence of a value for the game, and finding a pair of approximate optimal stopping times.
The methodology to prove these results is nowadays called the \emph{martingale approach}, initiated by Snell \cite{snell} for the optimal stopping problem (see the presentation of the martingale and Markovian approaches to optimal stopping in \cite{peskir2008}).
In the very same year, two related publications appeared. 
First, 
a game formulated for Markov chains by Frid in \cite{frid1969},
and second, 
one formulated for the Wiener process by Gusein-Zade in \cite{gusein1969}.
In these last two papers, each player has a set where he can stop the process, 
being these two sets disjoint.

The existence of the value of these games is further considered in different frameworks.
To mention two different approaches, Alario-Nazaret et al. establish, for optional processes, the existence of a saddle point under the so called ``Mokobodzki's assumption'' \cite[Thm. 2.3.]{alarionazaret}, and Cvitani\'c and Karatzas \cite{cvitanickaratzas},
construct the value as the solution of a doubly reflected backward stochastic differential equation, in this case the underlying process being a diffusion.

Our formulation, initiated by Neveu \cite[VI-6]{neveu},
is the stochastic game proposed by Ekstr\"om and Peskir \cite{ekstrom2008}
(see also Peskir \cite{peskir2008}). 
More precisely, Ekstr\"om and Peskir \cite{ekstrom2008} prove, 
for a strong Markov process, with right continuous and quasi-left-continuous trajectories,
under the order condition $G_1\leq G_2$ and classical integrability conditions, that the game has a value and also a Nash equilibrium, where the stopping times are the hitting times of the sets where the value coincides with the respective payoff functions.
This fine result is achieved in several steps, based first on the martingale approach, using the Wald-Bellman equations 
and afterwards the Markovian approach.
The Nash equilibrium strategies for each player are stopping times, 
respectively obtained as hitting times of certain Borel sets.
Our purpose is then, in this framework, to find explicit solutions when the underlying Markov process is a general L\'evy process.


For games when the order condition does not hold, see for instance 
Touzi and Vieille \cite{touzivieille},
and more recently Christensen and Schultz \cite{christensenschultz}.
Games where the integrability condition does not hold, in the framework of positive diffusions, are considered by Ekstr\"om and Villeneuve \cite{ekstromvilleneuve}, where interesting examples are provided.
As in the case of optimal stopping problems, there is a strong motivation coming from mathematical finance. 
In the case of a game, the max player is the holder of an American type contract 
(like in the optimal stopping situation), 
meanwhile the min player is the issuer of the contract, that, paying a penalty, 
can also decide the time of execution of the American contract. 
This contract receives  the name of \emph{callable} American contracts.
In this direction, we can mention the work of Kifer \cite{kifer_2000},
who studies the non-arbitrage pricing of such a contract, 
for both the Cox-Ross-Rubinstein binomial model and the Black-Scholes model;
Kyprianou \cite{kyprianou_2004}, 
where explicit solutions for put callable perpetual options are found; 
and Emmerling \cite{emmerling_2012}, where explicit solutions of callable call options are provided. 
These last two  works consider the Black-Scholes model.



When considering particular classes of processes, a situation in which one expects to obtain more explicit solutions to particular problems,
there are important contributions when the underlying process is a diffusion.
In addition to the aforementioned results of Kyprianou \cite{kyprianou_2004} and Emmerling \cite{emmerling_2012}, 
Alvarez \cite{alvarez} relies on the fundamental solutions of the second-order differential equation associated with the generator of the diffusion process in order to find the value of a game and characterize equilibrium points in some particular situations.
In this respect, 
it is well known that in order to find explicit solutions to \emph{boundary} problems for a diffusion process,
such as computing hitting barrier probabilities, finding a stationary distribution of a reflected diffusion within an interval, 
solving an optimal stopping problem, 
or solving a stopping game as above, the fundamental solutions often contain all the necessary information to solve the aforementioned problems.



However, the generalization of these ideas or procedures to L\'evy processes encounter significant difficulties. 
The key problem is the overshoot: boundary conditions in L\'evy processes are not limited to one or two points (the extremes of the interval where the process evolves) but extend over the entire line, as the process, when stopped, can potentially reach any point.
In analytic terms, many boundary problems for diffusions can be approached by solving an ordinary differential equation
on an interval or a half-line,
whereas the same type of problem for a L\'evy process requires to solve an integro-differential equation, 
which, in general, does not have explicit solutions.


At least two strategies have been proposed to face this difficulty. 
An approach is to consider L\'evy processes with one-sided jumps, 
called \emph{spectrally} (say negative) L\'evy processes. 
This fact simplifies the boundary conditions when the process does not jump over a boundary.
Even in the case when the process does overshoot, 
the exponential nature of the distribution of one of the Wiener-Hopf factors (the maximum in the case of negative jumps) can help to obtain useful information regarding the solution of the boundary problem,
as the optimal threshold for the American put option for spectrally negative L\'evy process found by Chan \cite{chan}. 
In many of these cases, the scale function --which is computable for large classes of these processes (see Hubalek and Kyprianou \cite{hk2010} and Kuznetsov \cite{kkr})-- provides valuable information for solving some problems.
Examples of this situation are the Dynkin games for spectrally L\'evy processes considered by Bardoux and Kyprianou in \cite{bardoux} (the McKean game) and \cite{bardoux2} (the Shepp-Shiryaev game). 

A second approach leverages the memoryless property of exponential random variables: when the jump is exponential, the overshoot is independent of the departure point, potentially allowing for closed-form solutions to various boundary problems. 
Based on this principle, it is possible to consider boundary problems for L\'evy processes with jump distributions to include phase-type distributions 
\cite{asmussen},
or those with rational transforms \cite{lm},
or even when the  density of the L\'evy measure is given by an infinite series of exponential functions with positive coefficients, as considered in \cite{kuznetsov:2010}.
Examples of closed solutions for optimal stopping of L\'evy processes with exponential-like jumps can be found, for instance in 
\cite{asmussen},
\cite{kou},
or
\cite{mordecki}.
Closed solutions for Dynkin games in the context of mathematical finance (i.e. the pricing of certain callable perpetual contracts) for processes with positive exponential jumps have been found in Gapeev and K\"uhn \cite{gapeev}.

At the same time, there have been efforts to solve optimal stopping problems for general L\'evy processes using the distributions of the supremum and infimum of the process. 
The first work addressing one-sided optimal stopping problems through the distribution of the supremum
was performed for general random walks by Darling et al. \cite{darling},
and generalized to L\'evy processes by Mordecki \cite{mordecki} (see also \cite{alili}).
Posterior developments in optimal stopping and the supremum can be found in the monograph by Kyprianou \cite{Kyprianou}. 
For two-sided problems, both the infimum and the supremum are expected to participate in the solution. 
A verification theorem for L\'evy processes (as a particular case of Hunt processes) 
is presented in \cite{mordecki-salminen}. 
Another verification theorem for the two-sided case was presented in \cite{CS}.
A third verification theorem for L\'evy processes, where also examples with closed solutions are included, 
was provided by Mordecki and Oli\'u  \cite{Oliu}.
In a parallelism with diffusions, one can think that the distributions of the infimum and the supremum of a L\'evy process play the role of the fundamental solutions of a diffusion, i.e., they contain all the information necessary to solve boundary problems, and, if known, allow for an explicit solution of the problem.


As we mentioned, in the present paper we analyze a Dynkin game in the formulation of Ekstr\"om and Peskir \cite{ekstrom2008}, 
with the particularity that the underlying process is a L\'evy process.
Although the two-sided optimal stopping (with only one max-player)
is an optimization problem,
and a Dynkin game with two players requires to find an equilibrium, i.e. it is a minimax-problem,
we found that some ideas of \cite{Oliu} are applicable,
and a (somewhat similar) verification theorem is obtained.
More precisely, 
the verification theorem proposes to construct the value function in terms of 
the infimum and supremum of the L\'evy process, 
with the help of  two \emph{averaging} functions,
which should be found.
The optimal stopping rules are the support of the respective averaging functions.
Naturally, as the distribution of the infimum and the supremum of the process are involved,
the Wiener-Hopf factorization obtained by Rogozin \cite{rogozin}  
for L\'evy processes (see also \cite{Bertoin} or \cite{Kyprianou}), 
appears as a natural tool to find explicit solutions for Dynkin games under this approach. 
%

We afterwards apply this verification result in two situations.
We first consider three 
examples with symmetric affine payoff functions and L\'evy processes with known Wiener-Hopf factors,
finding explicit solutions.
The processes considered are the Brownian motion with drift, 
the Cram\'er-Lundberg process with exponential jumps, 
and the compound Poisson process with double sided exponential jumps. 
These three examples share analytical properties, 
and an explicit solution to a symmetric game is presented 
(symmetric games have not been considered in the literature).
Secondly, we present a financial application, by the consideration of a perpetual callable futures contract. It is a classical perpetual contract, where not only the holder can select the execution time, but the issuer can also choose it, giving rise to a Dynkin game. The stock in this financial problem is modeled through a Kou's process \cite{kou}.



The rest of the paper is organized as follows. 
Section \ref{section:preliminaries} includes some preliminaries, 
necessary to introduce in Section  \ref{section:main} the main verification theorem.
This section also includes the proof of this main result, 
and the computation of the lateral derivatives of the value function at the contact points, 
in terms of the averaging functions and the possible mass of the supremum 
(or infimum) of the process at the origin.
Section \ref{section:applications} presents an application.
The Dynkin game is driven by a process with known Wiener-Hopf factors,
and the payoff functions are affine.
For three processes (Brownian motion with drift, Cramér-Lundberg and Compound Poisson) we obtain the explicit solutions of the considered Dynkin game, 
and illustrate each situation with numerical examples. 
The smooth pasting property is analyzed in each example, as in some cases it does not hold.
Section \ref{section:callable} presents the computation of optimal strategies 
and value function in a callable perpetual futures contract, 
with an underlying Kou's process \cite{kou}.
In Section \ref{sec:conclusion} some conclusions are presented.


\section{Preliminaries}\label{section:preliminaries}
Let $X=\{X_t\colon t\geq 0\}$ be a L\'evy process defined on a
stochastic basis ${\cal B}=(\Omega, {\cal F}, {\bf F}=({\cal
F}_t)_{t\geq 0}, \P_x)$ departing from $X_0=x$. 
Assume that the filtration ${\bf F}$ satisfies the usual conditions (see \cite{js}).
The corresponding expectation is denoted by $\E_x$,
and for short we denote $\E=\E_0$ and $\P=\P_0$. 
The L\'evy-Khintchine formula characterizes the law of the process,
stating, for $z\in i\R$, that 
$
\E e^{zX_t}=e^{t\Psi(z)}
$
with
\begin{equation*}
\Psi(z)=cz+\frac{\sigma^2}{2}z^2+\int_{\R}\left(e^{zy}-1-zy\indicator_{\{|y|<1\}}\right)\Pi(dy),
\end{equation*}
where $c\in\R$, $\sigma\geq 0$ and $\Pi(dy)$ is a non-negative measure (the jump measure) 
that satisfies the integrability condition $\int_{\R}(1\wedge y^2)\Pi(dy)<\infty$.
For general references on L\'evy processes see \cite{Bertoin} or \cite{Kyprianou}. 
Given the stochastic basis ${\cal B}$ the set of stopping times is the set of random variables
$$
\M=\{\tau\colon\Omega\to[0,\infty] \text{ such that } \{\tau\leq t\}\in\mathcal{F}_t \text{ for all $t\geq 0$}\}.
$$
Observe that we allow the possibility $\tau=\infty$, 
as several stopping rules that participate in the solution of the considered problems are within this class.

The Wiener-Hopf factorization relates the characteristic functions of the overall infimum and supremum of a L\'evy process, stopped at an independent exponential time,
with the characteristic exponent given by the L\'evy-Khintchine formula \cite[Chapter VI]{Bertoin}.  
These two random variables play a key role in the solution of two-sided problems
and are defined respectively by
\begin{equation}\label{eq:supinf}
I=\inf\{X_t\colon 0\leq t\leq e_r\}
\qquad\text{and}\qquad
S=\sup\{X_t\colon 0\leq t\leq e_r\},
\end{equation}
where $e_r$ is an exponential random variable of parameter $r>0$, independent of $X$.
Observe that as $r>0$, both random variables $S$ and $I$ are proper. 
It is also worth saying that the parameter $r$ plays the role of a discount factor in Definition \ref{def:dg}
of the Dynkin game, and remains the same constant throughout the paper.
Then, Wiener-Hopf factorization (see \cite{rogozin} or \cite[Theorem VII.5]{Bertoin}) 
states
\begin{equation}
    \label{eq:wh}
\frac{r}{r-\Psi(z)}=\E(e^{zS})\E(e^{zI}).
\end{equation}
This equation allows, in some cases, to compute the distribution of the supremum and/or the infimum of a L\'evy process (see for instance \cite{kuznetsov} or \cite{lm}).

Our definition of a Dynkin game requires two technical conditions, that we present below. Consider $G_i\colon\R\to\R$ $(i=1,2)$, a L\'evy process $X$, 
and a positive discount factor $r$. 

The first condition (see \cite[eq.(2.1)]{ekstrom2008}) is the integrability condition
\begin{equation}\label{eq:ic}
\E_x\left(\sup_{t\geq 0}\left|e^{-rt}G_i(X_t)\right|\right)<\infty,\quad\text{for $i=1,2$}.
\end{equation}
The second condition involves the behavior at infinity, 
requiring that
\begin{equation}\label{eq:limit}
    \lim_{t\to\infty}e^{-rt}G_i(X_t)=0,\quad i=1,2.
\end{equation}
Then, for an arbitrary stopping time $\tau\in\M$, we define
$$
e^{-r\tau}G_i(X_\tau)\indicator_{\{\tau=\infty\}}=0, \quad i=1,2,
$$
that is consistent in the sense that
\begin{equation}\label{eq:infinity}    
\lim_{t\to\infty}\E_x\left(e^{-r(\tau\wedge t)}G_i(X_{\tau\wedge t})\right)=
\E_x\left(e^{-r\tau}G_i(X_{\tau})
\indicator_{\{\tau<\infty\}}\right), \quad i=1,2,
\end{equation}
where the limit is obtained by dominated convergence using \eqref{eq:ic} and \eqref{eq:limit}. 

\begin{definition}[Dynkin game]\label{def:dg}
Let 
$G_1(x)\leq G_2(x)$
be two real continuous payoff functions.
Consider a L\'evy process $X$ and a discount factor $r>0$, 
such that conditions \eqref{eq:ic} and \eqref{eq:limit} above hold.
Given two stopping times $\sigma$ and $\tau$, define the expected payoff
$$
\Mx(\sigma,\tau)=\E_x\left(
e^{-r\tau}G_1(X_{\tau})\indicator_{\{\tau\leq\sigma\}}
+
e^{-r\sigma}G_2(X_{\sigma})\indicator_{\{\sigma<\tau\}}
\right).
$$
The Dynkin game (DG) problem consists
in finding the value function $V(x)$ and two optimal stopping rules ${\sigma^*}$ and ${\tau^*}$ such that
\begin{align}
V(x)&=\inf_{\sigma}\sup_{\tau}\Mx(\sigma,\tau)=\sup_{\tau}\inf_{\sigma}\Mx(\sigma,\tau)=\Mx(\sigma^*,\tau^*).
\label{eq:dg}
\end{align}
\end{definition}
It is important to mention that, 
in terms of game theory, there are two players, 
the first one is the max player, that chooses a strategy $\tau$; 
the second one is the min  player,
and chooses a strategy $\sigma$.  As a result of these elections, the min player pays 
$$
e^{-r\tau}G_1(X_{\tau})\indicator_{\{\tau\leq\sigma\}}
+
e^{-r\sigma}G_2(X_{\sigma})\indicator_{\{\sigma<\tau\}},$$
to the max player. In this context the pair of optimal strategies $(\sigma^*,\tau^*)$ is called a Nash equilibrium.

The following simple result presents inequalities useful for solving a Dynkin game
(see \cite[(1.4)]{ekstrom2008}).   

\begin{proposition}\label{lemma:dg}
Consider a pair $(\tilde{\sigma},\tilde{\tau})$ of stopping times  for the problem  in Definition \ref{def:dg} and a real valued function $\widetilde{V}(x)$, defined by 
\begin{align}
\widetilde{V}(x)&=\Mx(\tilde{\sigma},\tilde{\tau}).\tag{MG}\label{eq:mg}
\end{align}
Assume that the two following conditions hold
\begin{align}
\widetilde{V}(x)&\geq\Mx(\tilde{\sigma},\tau),
\quad\text{ for all $\tau\in\M$},\tag{SPMG}\label{eq:spmg}\\
&\notag\\
\widetilde{V}(x)&\leq\Mx(\sigma,\tilde{\tau}),
\quad\text{ for all $\sigma\in\M$}.\tag{SBMG}\label{eq:sbmg}
\end{align}
Then, the DG in Definition \ref{def:dg}
has value function $\widetilde{V}(x)$, and $(\tilde{\sigma},\tilde{\tau})$  is a  pair of optimal stopping times  for the problem, meaning that \eqref{eq:dg} holds.
\end{proposition}



\section{Main result}\label{section:main}
In this section we present a verification theorem that constructs the solution of a Dynkin game in terms of the supremum and infimum of the process, 
combined with two (unknown) averaging functions.
\begin{theorem}
\label{theorem:1}
Consider a L\'evy process $X$,
a discount rate $r>0$, and two continuous reward functions $G_1\leq G_2\colon\R\to\R$.
Assume that there exist two points 
 $x_I<x_S$
and two continuous non-decreasing functions: 
$Q_I$ s.t. $Q_I(x)=0$ for  $x_I\leq x$
and
$Q_S$ s.t. $Q_S(x)=0$ for $ x\leq x_S$
(named \emph{averaging functions});
and such that
\begin{equation}\label{eq:equal}
\E_x Q_I(I)+\E_x Q_S(S)=
\begin{cases}G_1(x),&\quad\text{when $x_S\leq x$,}\\
\\
G_2(x),&\quad\text{when $x\leq x_I$,}\end{cases}
\end{equation}
where $I$ and $S$ are defined in \eqref{eq:supinf}.
Define the function
\begin{equation}\label{eq:ve}
V(x)=
\E_x Q_I(I)+ \E_x Q_S(S), \text{ for all $x\in\R$.}\end{equation}
Then, if the condition
\begin{equation}\label{eq:geq}G_1(x)\leq V(x)\leq G_2(x),\end{equation}
holds for all $x\in[x_I,x_S]$, the DG  of Definition \ref{def:dg}
has value function $V(x)$ in \eqref{eq:ve}, and optimal stopping rules given by
\begin{equation}\label{eq:tau0}\sigma^*=\inf\{t\geq 0\colon X_t\leq x_I\}\qquad\text{and}\qquad\tau^*=\inf\{t\geq 0\colon X_t\geq x_S\}.\end{equation}
\end{theorem}
\begin{remark}
Observe that, taking into account the asymptotic behavior of a  L\'evy process 
(see Thm. VI.12 in \cite{Bertoin}) the stopping times in \eqref{eq:tau0} satisfy
$\P(\sigma^*\wedge\tau^*<\infty)=1$.
Furthermore, in view of \eqref{eq:infinity} and condition \eqref{eq:geq},
for an arbitrary stopping time in $\M$,
we have
\begin{equation}
\label{eq:V_infinity}    
\E_x\left(e^{-r\tau}V(X_{\tau})
\indicator_{\{\tau=\infty\}}\right)=0.
\end{equation}
\end{remark}

Despite the fact that we do not require specific conditions on $G_1$ and $G_2$ some simple arguments show necessary conditions in order to be able to apply the Theorem above.
\begin{corollary}
If Theorem \ref{theorem:1} holds  
we obtain the following necessary conditions on $G_1$ and $G_2$.
\begin{itemize}
\item[\rm(i)]
$G_1(x) \text{ is increasing for $x\geq x_S$}$,
\item[\rm(ii)] 
$G_2(x) \text{ is increasing for $x\leq x_I$}$,
\item[\rm(iii)]
$\lim_{x\to-\infty}G_2(x)<0$,
\item[\rm(iv)]
$\lim_{x\to+\infty}G_1(x)>0$,
\item[\rm(v)] $G_1(x)<G_2(x)$ for $x\in\mathbb{R}.$ 
\end{itemize}

\begin{proof}
From the representation \eqref{eq:ve}, by dominated convergence, taking into account the signs and monotonicity properties of $Q_I$ and $Q_S$, conditions \rm(i)-\rm(iv) follows.
Condition \rm(v) is obtained as optimal stopping sets are disjoint.
\end{proof}
\end{corollary}

\subsection{Proof of Theorem \ref{theorem:1}}
The proof of the main result is carried out in the following two Lemmas.
\begin{lemma}[MG]\label{lemma:1}
Consider a DG as in Definition \ref{def:dg}.
Assume that there exist points $x_S$, $x_I$ and averaging functions $Q_S,Q_I$  that satisfy the hypothesis of Theorem \ref{theorem:1}.
Then, for $V$  and $(\tau^{*},\sigma^{*})$ defined by \eqref{eq:ve} and \eqref{eq:tau0}, the equality \eqref{eq:mg} holds.
\end{lemma}
\begin{proof}
Denote 
$U=\R\setminus(x_I,x_S)$. Define $G\colon U\to\R$ by
$$
G(x)=
\begin{cases}
G_1(x),&\quad\text{when $x_S\leq x$,}\\
\\
G_2(x),&\quad\text{when $x\leq x_I$.}
\end{cases}
$$
Denote $\rho^*=\sigma^*\wedge\tau^*$.
Observe that as $X$ is a L\'evy process, $\P_x(\rho^*<\infty)=1$.
As $X_{\rho^*}\in U$ and $V=G$ on $U$, we have
\begin{equation}\label{eq:vege}
\E_x(e^{-r{\rho^*}}G(X_{\rho^*}))=\E_x(e^{-r{\rho^*}}V(X_{\rho^*})).
\end{equation}
On the other side, by the monotonicity of the functions $Q_S$ and $Q_I$,
$$
\aligned 
V(x)&=\E_x Q_I(I)+\E_x Q_S(S)=\E_x\left(
Q_I\left(\inf_{0\leq t\leq e_r}X_t
\right)+
Q_S\left(\sup_{0\leq t\leq e_r} X_t
\right)
\right)\\ &
=\E_x\left(
\inf_{0\leq t\leq e_r}Q_I(X_t)
+
\sup_{0\leq t\leq e_r}Q_S(X_t)
\right).
\endaligned
$$

Observe that if $e_r<{\rho^*}$ we have $x_I<I\leq S<x_S$ (i.e. the process still lives within $[x_I,x_S]$), 
and then 
\begin{equation*}\label{eq:zero}
\inf_{0\leq t\leq e_r}Q_I\left(X_t\right)=\sup_{0\leq t\leq e_r}Q_S\left( X_t\right)=0,
\end{equation*}
because $Q_I(x)=Q_S(x)=0$ in $[x_I,x_S]$.
Also note, that
\begin{equation*}
\inf_{0\leq t<\rho^*}Q_I\left(X_t\right)=\sup_{0\leq t<\rho^*}Q_S\left(X_t\right)=0.
\end{equation*}
In conclusion, 
\begin{align*}
\E_xQ_I(I)&=\E_x\left(\inf_{0\leq t\leq e_r}Q_I(X_t)\right)=\E_x\left(\inf_{{\rho^*}\leq t\leq e_r}Q_I(X_t)\right)=:q_I(x),\\
&\\
\E_xQ_S(S)&=\E_x\left(\sup_{0\leq t\leq e_r}Q_S(X_t)\right)=\E_x\left(\sup_{{\rho^*}\leq t\leq e_r}Q_S(X_t)\right)=:q_S(x),
\end{align*}
and we have
$$
V(x)=q_I(x)+q_S(x).
$$
Consider now $\tilde{X}=\{\tilde{X}_s=X_{{\rho^*}+s}-X_{\rho^*}\colon s\geq 0\}$ 
that, by the strong Markov property,
is independent of $\F_{\rho^*}$ and has the same distribution as $X$, 
and denote by $\tilde{\E}_x$ the expectation w.r.t. $\tilde{X}$.
Based on these considerations, we have
\begin{align}
q_I(x)&=\E_x {Q_I}(I)=\E_x\left(\sup_{{\rho^*}\leq t\leq e_r}{Q_I}(X_t)\right)\notag\\
&=\E_x\left(\int_{\rho^*}^{\infty}\sup_{{\rho^*}\leq t\leq u}{Q_I}(X_t)re^{-ru}du\right)\label{eq:uu}\\
&=\E_x\left(e^{-r{\rho^*}}\int_0^{\infty}\sup_{{\rho^*}\leq t\leq {\rho^*}+v}{Q_I}(X_t)re^{-rv}dv\right)\label{eq:v}\\
&=\E_x\left(e^{-r{\rho^*}}\int_0^{\infty}\sup_{{\rho^*}\leq t\leq {\rho^*}+v}{Q_I}(X_{\rho^*}+X_t-X_{\rho^*})re^{-rv}dv\right)\label{eq:t}\\
&=\E_x\left(e^{-r{\rho^*}}\int_0^{\infty}\sup_{0\leq s\leq v}{Q_I}(X_{\rho^*}+X_{{\rho^*}+s}-X_{\rho^*})re^{-rv}dv\right)\label{eq:s}\\
&=\E_x\left(e^{-r{\rho^*}}\int_0^{\infty}\sup_{0\leq s\leq v}{Q_I}(X_{\rho^*}+\tilde{X}_s)re^{-rv}dv\right)\notag\\
&=\E_x\left(e^{-r{\rho^*}}\tilde{\E}_{X_{\rho^*}}\left[\int_0^{\infty}\sup_{0\leq s\leq v}{Q_I}(\tilde{X}_s)re^{-rv}dv\right]\right)\notag\\
&=\E_x\left(e^{-r{\rho^*}}\tilde{\E}_{X_{\rho^*}}\left[\sup_{0\leq s\leq e_r}{Q_I}(\tilde{X}_s)\right]\right)=\E_x\left(e^{-r{\rho^*}}q_I(X_{\rho^*})\right),\notag
\end{align}
where in \eqref{eq:uu}
we take expectation w.r.t. the independent variable $e_r$,
we change variables according to $v=u-{\rho^*}$ to pass from \eqref{eq:uu} to \eqref{eq:v},
and denote $s=t-{\rho^*}$ to pass from \eqref{eq:t} to \eqref{eq:s}. 
The same relation holds with $q_S$ and $Q_S$, i.e.:
 $$
q_S(x)=\E_x {Q_S}(S)=\E_x\left(e^{-r{\rho^*}}q_S(X_{\rho^*})\right).
$$
Summing up these two relations, in view of \eqref{eq:vege},
we obtain
$$
\aligned
V(x)	&=\E_x\left(e^{-r{\rho^*}}V(X_{\rho^*})\right)=\E_x\left(e^{-r{\rho^*}}G(X_{\rho^*})\right)\\
	&=\E_x\left(e^{-r{\tau^*}}G_1(X_{\tau^*})\indicator_{\{\tau^*\leq\sigma^*\}}+
    e^{-r{\sigma^*}}G_2(X_{\sigma^*})\indicator_{\{\sigma^*<\tau^*\}}\right),
\endaligned
$$ 
concluding the proof of the Lemma.
\end{proof}


\begin{lemma}[SPMG-SBMG]\label{lemma:spmg}
Consider the DG of Definition \ref{def:dg}.
Assume that there exist points $x_I$, $x_S$ and averaging functions $Q_I,Q_S$  that satisfy the hypothesis of Theorem \ref{theorem:1}.
Then, for $V$ defined by \eqref{eq:ve}, 
\begin{itemize}
\item[\rm(i)] the supermartingale inequality \eqref{eq:spmg} holds,
\item[\rm(ii)] the submartingale inequality \eqref{eq:sbmg} holds.
\end{itemize}
\end{lemma}


\begin{proof} Let us see (i).
As $Q_I(x)=0$ for $x\geq x_S$, in view of the definition \eqref{eq:tau0} of $\sigma^*$,
for an arbitrary stopping time $\tau\in\M$,
we have the following a.s. identity
\begin{equation}
\label{eq:inf}    
\inf_{0\leq t\leq e_r}Q_I(X_t)=
\inf_{\tau\wedge\sigma^*\leq t\leq e_r}Q_I(X_t)\indicator_{\{\tau\wedge\sigma^*\leq e_r\}}.
\end{equation}
On the other side, as $Q_S(x)\geq 0$, we have
\begin{equation}
\label{eq:sup}    
\sup_{0\leq t\leq e_r}Q_S(X_t)\geq 
\sup_{\tau\wedge\sigma^*\leq t\leq e_r}Q_I(X_t)\indicator_{\{\tau\wedge\sigma^*\leq e_r\}}.
\end{equation}
With these relations in view, we obtain
\begin{align}
    V(x)  &=\E_x Q_I(I)+\E_x Q_S(S)\notag\\
    &=\E_x \left(
    \inf_{0\leq t\leq e_r} Q_I(X_t)
    + 
    \sup_{0\leq t\leq e_r} Q_S(X_t)
    \right)\label{eq:l2}\\
 &\geq\E_x\left[\left(
 \inf_{\tau\wedge\sigma^*\leq t\leq e_r}Q_I(X_t) 
 +
 \sup_{\tau\wedge\sigma^*\leq t\leq e_r} Q_S(X_t)
 \right)\indicator_{\{\tau\wedge\sigma^*<e_r\}}\right]\label{eq:l3}\\ 	
 &=\E_x\left[\int_{\tau\wedge\sigma^*}^\infty\left(
 \inf_{\tau\wedge\sigma^*\leq t\leq v} Q_I(X_t) 
 + 
 \sup_{\tau\wedge\sigma^*\leq t\leq v} Q_S(X_t)
 \right)re^{-rv}\,dv\right]\label{eq:l4}\\
&\label{eq:l5}
=\E_x\left[e^{-r(\tau\wedge\sigma^*)} V(X_{\tau\wedge\sigma^*})\indicator_{\{\tau<\infty\}}\right]
=
\E_x\left[e^{-r(\tau\wedge\sigma^*)} V(X_{\tau\wedge\sigma^*})\right],
\end{align}
where we pass from \eqref{eq:l2} to \eqref{eq:l3} in view of \eqref{eq:inf} and \eqref{eq:sup},
from line \eqref{eq:l3} to \eqref{eq:l4} integrating with respect to the exponential density of the independent
random variable $e_r$, and the passage from \eqref{eq:l4} to \eqref{eq:l5}
is the same procedures as in \eqref{eq:uu} to \eqref{eq:s};
and, the last equality uses \eqref{eq:V_infinity}. 
Finally, taking into account our hypothesis \eqref{eq:geq},
and the fact that 
$V=G_2$ for $x\leq x_I$, 
we obtain
$$
e^{-r(\tau\wedge\sigma^*)} V(X_{\tau\wedge\sigma^*})
\geq e^{-r\tau}G_1(X_\tau)\indicator_{\{\tau\leq\sigma^*\}}+ e^{-r\sigma^*}G_2(X_{\sigma^*})\indicator_{\{\sigma^*<\tau\}}
.$$
Taking expectations, the inequality \eqref{eq:spmg} is proved, concluding the proof of (i).
The proof of (ii) is analogous.
\end{proof}
\begin{proof}[Proof of Theorem \ref{theorem:1}] 
The proof is an application of Proposition \ref{lemma:dg}, 
where Lemmas
\ref{lemma:1} gives \eqref{eq:mg},
and \ref{lemma:spmg} gives the conditions \eqref{eq:spmg} and \eqref{eq:sbmg}.
\end{proof}
\subsection{Smooth pasting}\label{subsection:smooth}
One of the interesting features of Theorem \ref{theorem:1} is that, although it is a verification result in the sense that it requires a conjecture to identify a solution (a common characteristic in many optimal stopping problems), it does not rely on the smooth pasting condition. This condition is typically used in stopping problems within the diffusion framework \cite{peskir2006}. This distinction is significant because, as illustrated in our examples and well documented in the context of L\'evy processes, the smooth pasting condition does not consistently hold 
(see for instance the discussion in \cite{alili} and the examples in \cite{Kyprianou}).
Due to its similarity, 
we adopt the approach of \cite{Oliu}, 
where the following result is obtained.
\begin{proposition}[Theorem 3.1 in \cite{Oliu}]\label{proposition:sp}
Consider a function of the form
\begin{equation}\label{eq:w}
W(x)=\E_xQ(S).
\end{equation}
Assume that the three following conditions hold:
\vskip1mm\par\noindent
{\rm(i)}
$Q(x)=0$, for $x\leq x_0$ and $Q(x)$ is non-decreasing and in class $C^2[x_0,\infty)$,
\vskip1mm\par\noindent
{\rm(ii)} there exist $A>0$ and $\alpha>0$ s.t. 
$|Q''(x)|\leq A e^{\alpha x}$,
\vskip1mm\par\noindent
{\rm(iii)} for $\alpha>0$ above, it holds
$
\E e^{\alpha X_1}<e^r.
$
\vskip1mm\par\noindent
Then
\begin{align}
\label{eq:atom}
W'(x_0+)-W'(x_0-)&=Q'(x_0+)\P(S=0),\\
W'(x+)-W'(x-)&=0,\text{ for $x>x_0$}. \label{eq:atom2}
\end{align}
\end{proposition}
To apply this result in our situation, observe that
the function $V(x)$ is the sum of two functions of the form
\eqref{eq:w}, because 
$\E_x Q_I(I)=\E_{-x}Q_I(-\hat{S})$ 
where $\hat{S}$ is the supremum of
the dual process $\hat{X}=-X$. 
In conclusion, 
if the averaging functions $Q_I,\ Q_S$ meet conditions (i) to (iii), and the payoffs $G_1,G_2$ 
are differentiable, we will have smooth pasting
at $x_I$ if and only if the condition $\P(I=0)=0$ holds, 
and smooth pasting
at $x_S$
if and only if the condition $\P(S=0)=0$  holds. 
This will be discussed more in detail in the examples of next section.

\subsection{Integrability}
To end the section, we provide the following result in order to check conditions \eqref{eq:ic} and \eqref{eq:limit}. Observe that, unless the functions $G_i$ are bounded, condition $r>0$ is strictly necessary for \eqref{eq:ic} to hold true.
\begin{proposition} \label{proposition:integrability}
Consider a L\'evy process $X$,
a discount rate $r>0$,  a function $G\colon\R\to\R$, 
and assume that there exist positive constants $A,B,\alpha$ such that
\begin{equation}\label{eq:cosh}
|G(x)|\leq A+B\cosh(\alpha x),    
\end{equation}
and    
\begin{equation}\label{eq:alpha}
    \E e^{\alpha X_1}< e^r\qquad\text{and}\qquad \E e^{-\alpha X_1}< e^r.
\end{equation}
Then, the function $G$ satisfies conditions \eqref{eq:ic} and \eqref{eq:limit}.
\end{proposition}
\begin{proof} 
In view of \eqref{eq:alpha}, the characteristic exponent $\Psi(z)$ has an analytic extension to the strip 
$\Re(z)\in[-\alpha,\alpha]$ in the complex plane (see \cite[Thm. 25.17]{sato}). 
From this follows that $\E|X_1|<\infty$
and that $\E X_1=\Psi'(0)$.
Consider the L\'evy process $\{\bar X_t=\alpha X_t-rt\colon t\geq 0\}$.
Its characteristic exponent $\bar\Psi(z)$
satisfies $\bar\Psi(1)=\Psi(\alpha)-r<0$ (the first statement in \eqref{eq:alpha}).
From this follows $\E\bar X_1<0$, meaning that $\lim_{t\to\infty} (\alpha X_t-rt)=-\infty$.
A similar reasoning based on the second statement in \eqref{eq:alpha} implies 
$\lim_{t\to\infty} (-\alpha X_t-rt)=-\infty$. In view of \eqref{eq:cosh}, condition 
\eqref{eq:limit} holds.

In order to see  now condition \eqref{eq:ic}
we use \cite[Lemma 1]{mordecki} that can be stated 
for the L\'evy process $\{\alpha X_t-rt\colon t\geq 0\}$ without discount,
as the equivalence of the following two statements:
\begin{align}
\E(e^{\sup_{t\geq 0}(\alpha X_t-rt)})&<\infty,\label{eq:(a)}\\
\E(e^{\alpha X_1-r})&<1. \label{eq:(b)}
\end{align}
Furthermore
\begin{equation}\label{eq:bound}
    \sup_{t\geq 0}e^{-rt}|G(X_t)|\leq A+\frac{B}2e^{\sup_{t\geq 0}(\alpha X_t-rt)}+
\frac{B}2e^{\sup_{t\geq 0}(-\alpha X_t-rt)}.
\end{equation}
As condition \eqref{eq:(b)}  is equivalent to the first statement in \eqref{eq:alpha}
(the same for $-\alpha$) in view of condition \eqref{eq:(a)} and \eqref{eq:bound},
the statement \eqref{eq:ic} follows.
\end{proof}

\section{Three processes with affine payoff}
\label{section:applications}

In this section we  consider three examples that are similar from an analytic point of view: Brownian motion with drift, Cram\'er-Lundberg process with exponential claims and Compound Poisson process with double sided exponential jumps. 
These examples share the structure of characteristic function for the infimum and the supremum.
%
%
More precisely, we assume that the infimum $I$  and the supremum $S$ in \eqref{eq:supinf}
have characteristic functions given respectively by
\begin{align}
\E e^{zI}&=\pi_I+(1-\pi_I)\frac{r_I}{r_I+z},\label{eq:sup-fc}\\
\E e^{zS}&=\pi_S+(1-\pi_S)\frac{r_S}{r_S-z},\label{eq:inf-fc}
\end{align}
where $0\leq \pi_I,\pi_S\leq 1$ are probabilities, and $-r_I<0<r_S$ are the two roots of the equation
$$
\Psi(z)=r.
$$ 
It is important to note that the  Brownian motion with drift is obtained when $\pi_I=\pi_S=0$;
the Cram\'er-Lundberg process is obtained when $\pi_I>0$ and $\pi_S=0$,
and  the Compound Poisson process is obtained when $\pi_I>0$  and $\pi_S>0$. 
Moreover, other processes with rational transform distributed jumps can be hopefully treated with a similar methodology, at the cost or more involved computations (see \cite{lm}).

In view of \eqref{eq:sup-fc} and \eqref{eq:inf-fc}, the random variables $-I$ and $S$ have (possibly) defective exponential distributions
with respective parameters $r_I$ and $r_S$,  and atoms of size $\pi_I$, $\pi_S$. Then, 
with a slight abuse of notation,
the respective densities can be denoted by
\begin{align}
f_I(x)&=\pi_I\delta_0(x)+(1-\pi_I)r_Ie^{r_Ix},\quad x\leq 0,\label{eq:inf_density}\\
f_S(x)&=\pi_S\delta_0(x)+(1-\pi_S)r_Se^{-r_Sx},\quad x\geq 0,\label{eq:sup_density}
\end{align}
where $\delta_0(x)$ denotes the Dirac mass measure at $x=0$.

For simplicity to obtain tractable solutions, we assume
\begin{equation}
\label{eq:g12}
G_1(x)=x-\delta\qquad\text{and}\qquad G_2(x)=x+\delta,
\end{equation}
for some $\delta>0$.

Observe now that the three process considered above (Brownian motion with drift, Cram\'er-Lundberg process with exponential claims and Compound Poisson process with double sided exponential jumps)
have a characteristic exponent that is analytic in a certain strip $-\alpha\leq\Re(z)\leq\alpha$,
and, for that $\alpha>0$ there exist positive constants $A,B$ such that
$
|G_i(x)|\leq A+B\cosh(\alpha x),\ (i=1,2).
$
This allows to apply Proposition \ref{proposition:integrability} 
to verify the conditions of Theorem \ref{theorem:1} in the considered examples.

For simplicity in notation, introduce the  constants $E_I, F_I, G_I, E_S, F_S, G_S$ that depend only on the parameters of the problem
\begin{align}
E_I&= -\E I=\dfrac{1-\pi_I}{r_I}>0, &E_S&= \E S=\dfrac{1-\pi_S}{r_S}>0,\label{eq:e12}\\
F_I&={1\over \E e^{-r_IS}}=\dfrac{r_I+r_S}{r_I+\pi_I r_S}>1, & F_S&={1\over \E e^{r_SI}}=\dfrac{r_I+r_S}{\pi_S r_I+r_S}>1,\label{eq:f12}\\
G_I&=F_I-1=\frac{(1-\pi_I)r_S}{r_I+\pi_Ir_S}>0, 
&G_S&=F_S-1=\frac{(1-\pi_S)r_I}{\pi_S r_I+r_S}>0.\label{eq:gf12}
\end{align}

\begin{theorem} \label{teo:ejtriple}
Under the assumptions \eqref{eq:inf_density} and \eqref{eq:sup_density}
for the densities of the respective infimum and supremum of the process $X$, and
for the linear payoffs $G_1(x)=x-\delta,$ $G_2(x)=x+\delta,$
for some $\delta>0$,
the DG of Definition \ref{def:dg} has optimal thresholds given by
\begin{align}
x_I&=-\delta-E_I+F_I\frac{E_Se^{r_Iu}-E_IG_S}{e^{(r_I+r_S)u}-G_IG_S},\label{eq:x1}\\
x_S&=\delta+E_S+F_S\frac{-E_Ie^{r_Su}+E_SG_I}{e^{(r_I+r_S)u}-G_IG_S},\label{eq:x2}
\end{align}
where $u$ is the unique solution of  
\begin{equation}
    u-2\delta-E_I-E_S
=\frac{-E_IF_Se^{r_Su}-E_SF_Ie^{r_Iu}+2E_SF_SG_I}{e^{(r_I+r_S)u}-G_IG_S}.
\label{eq:u}
\end{equation}
Moreover, the value function is given by
 \begin{equation}\label{eq:value}
V(x)=
\begin{cases}
x+\delta, & \text{if $x\leq x_I$},\\    
A_I e^{-r_I(x-x_S)}+A_S e^{r_S(x-x_I)},& \text{if $x_I\leq x\leq x_S$},\\
x-\delta, & \text{if $x\geq x_S$,}
\end{cases}
\end{equation}
with
\begin{equation}\label{eq:a12}
A_I=\frac{-E_Ie^{r_Su}+E_SG_I}{e^{(r_I+r_S)u}-G_IG_S},
\quad
 A_S=\frac{E_Se^{r_Iu}-E_IG_S}{e^{(r_I+r_S)u}-G_IG_S}.
\end{equation}
\end{theorem}

\begin{proof} 


In order to find functions $Q_I(x)$ and $Q_S(x)$ s.t. \eqref{eq:equal} holds,
we can compute
$\E_xQ_I(I)$ for $x\leq x_I$
and 
$\E_xQ_S(S)$ for $x\geq x_S$ 
   (using  \eqref{eq:inf_density} and  \eqref{eq:sup_density})
\begin{align}
q_I(x)=\E_xQ_I(I)&=\int_{(-\infty,0]}Q_I(x+y)f_I(y)dy\notag\\
&=\int_{(-\infty,0]}
\left({\pi_I}\delta_0(y)+(1-\pi_I)r_Ie^{r_Iy}\right) Q_I(x+y)dy\notag\\
&=\int_{-\infty}^{x\wedge x_I}
\left({\pi_I}\delta_0(z-x)+(1-\pi_I)r_Ie^{r_I(z-x)}\right)Q_I(z)dz\notag\\ 
&={\pi_I}Q_I(x)+\left(1-{\pi_I}\right)e^{-r_Ix}\int_{-\infty}^{x\wedge x_I} r_Ie^{r_Iz}Q_I(z)dz.
\label{q1}
\end{align}
Analogously 
\begin{equation}
q_S(x)=\E_xQ_S(S)
={\pi_S}Q_S(x)+\left(1-{\pi_S}\right)e^{r_Sx}\int_{x\vee x_S}^{+\infty} r_Se^{-r_Sz}Q_S(z)dz.
\label{q2}
\end{equation}
For convenience, we introduce the notation
\begin{align}
\relax A_I&=A_I(x_I,x_S)=(1-\pi_I)e^{-r_Ix_S}\int_{-\infty}^{x_I}Q_I(z) r_Ie^{r_Iz} dz,\label{A1}
\\ 
\relax A_S&=A_S(x_I,x_S)=(1-\pi_S)e^{r_Sx_I}\int_{x_S}^{+\infty}Q_S(z)r_Se^{-r_Sz}dz.\label{A2}
\end{align}
For the solution, we impose the condition \eqref{eq:equal} with the payoff functions of the problem. 
We first impose the continuity conditions at the optimal thresholds $x_I$ and $x_S$. We obtain the system 
\begin{equation}\label{eq:continuity}
\begin{cases}
e^{r_Iu}\relax  A_I+\relax  A_S=x_I+\delta,\\
\relax  A_I+e^{r_Su}\relax  A_S=x_S-\delta,
\end{cases}
\end{equation}
where we denoted 
$$
u=x_S-x_I.
$$
As we have jumps, the boundary conditions necessarily comprises the whole half lines 
$x\leq x_I$ 
and 
$x\geq x_S$, instead of only the boundary point.
This corresponds to the fact that the process can jump outside the interval $[x_I,x_S]$.
Accordingly we impose equation 
 \eqref{eq:equal} for $x\leq x_I$, that states
\begin{equation}\label{eq:ei1}
{\pi_I}Q_I(x)+\left(1-{\pi_I}\right)e^{-r_Ix}\int_{-\infty}^x r_Ie^{r_Iz}Q_I(z)dz+
e^{r_S(x-x_I)}\relax  A_S=x+\delta.
\end{equation}
The  relationship \eqref{eq:ei1} is an ordinary differential equation to find $Q_I$. 
If we multiply by $e^{r_Ix}$ both sides, assume enough regularity, and take derivatives, we have
\begin{align*}
{\pi_I}\left(r_Ie^{r_Ix}Q_I(x)+e^{r_Ix}Q_I'(x)\right)+& \left(1-{\pi_I}\right)r_Ie^{r_Ix}Q_I(x)+\\
&(r_I+r_S)e^{r_Ix}e^{r_S(x-x_I)}\relax  A_S
=r_I(x+\delta)e^{r_Ix}+e^{r_Ix}.
\end{align*}
Therefore, the function $Q_I(x)$ satisfies the differential equation
\begin{equation*}
{\pi_I\over r_I} Q'_I(x)+Q_I(x)=x+\delta+\frac{1}{r_I}-\left(1+\frac{r_S}{r_I}\right)\relax  A_Se^{r_S(x-x_I)}.
\end{equation*}
In  the case that $\pi_I=0$ we directly obtain the candidate $Q_I$ to use in the verification Theorem \ref{theorem:1}. When $\pi_I>0$,   we assume that $Q_I(x)\sim x$ when $x\to -\infty$, in order to find a candidate.
Solving  the differential equation according to this condition, 
with the introduced above notation, we find the function 
\begin{equation}
Q_I(x)=x+\delta+E_I-F_IA_Se^{r_S(x-x_I)},\quad{x\leq x_I}.
\label{eq:solP2}
\end{equation}
Similarly, for $x\geq x_S$, we obtain the differential equation
\begin{equation*}
{\pi_S}Q_S(x)+\left(1-{\pi_S}\right)e^{r_Sx}\int_x^{+\infty} r_Se^{-r_Sz}Q_S(z)dz+
e^{-r_I(x-x_{S})}\relax  A_I=x-\delta,
\end{equation*}
 with  solution 
\begin{equation}
Q_S(x)=x-\delta-E_S-F_SA_Ie^{-r_I(x-x_I)},\quad{x\geq x_S}.
\label{eq:solF2}
\end{equation}
Using now the functions \eqref{eq:solP2} and \eqref{eq:solF2}, 
we impose conditions $Q_I(x_I)=0$ and $Q_S(x_S)=0$ required in Theorem \ref{theorem:1},
to get
\begin{equation}\label{eq:boundary}
\begin{cases}
x_I+\delta+E_I-F_I\relax {A_S}=0,\\
x_S-\delta-E_S-F_S\relax {A_I}=0.
\end{cases}
\end{equation}
To solve the system above,  
we substitute   $x_I-\delta$ and $x_S+\delta$ from  \eqref{eq:continuity} 
obtaining the linear system
\begin{equation*} \label{eq:sistem_A}
\begin{cases}
e^{r_Iu}\relax {A_I}-G_I\relax {A_S}=-E_I,\\
-G_S\relax {A_I}+e^{r_Su}\relax {A_S}=E_S,
\end{cases}
\end{equation*}
that has solutions $A_I$ and $A_S$ given by the equation \eqref{eq:a12}.
We replace $\relax  A_I$ and $\relax  A_S$ in the system \eqref{eq:boundary}, we subtract both equations and  we obtain the equation \eqref{eq:u},
where we used $E_SF_SG_I=E_IF_IG_S$.

The left-hand side in \eqref{eq:u} is a linear increasing function in $u$,
with value at $u=0$ strictly less than $-(E_I+E_S)$, as $\delta>0$. 
The numerator of the right-hand side  (using the constants \eqref{eq:e12}, \eqref{eq:f12} and \eqref{eq:gf12}) results 
$$
-E_IF_Se^{r_Su}-F_IE_Se^{r_Iu}+2E_SF_SG_I\leq 0,
$$ 
which is a decreasing function in $u$
and the denominator of the right-hand side
is an increasing function in $u$. 
These properties show that the right-hand side of \eqref{eq:u} is a negative
increasing function to zero and when $u=0$ satisfies that it is equals $-(E_I+E_S)$. Therefore, we conclude that  \eqref{eq:u} has a unique positive solution for all $\delta>0$. 

Finally, with this value of the root $u$, from \eqref{eq:boundary} we find that the optimal strategies $x_I$ and $X_S$ are  \eqref{eq:x1} and \eqref{eq:x2}.
In order to compute $V(x)$ for $x\in[x_I,x_S]$, we refer to \eqref{eq:ve}. 
Observe that $Q_I(x)=Q_S(x)=0$
within this interval, so, taking into account \eqref{q1}-\eqref{q2} and \eqref{A1}-\eqref{A2}, we obtain  \eqref{eq:value}.
\end{proof}

\subsection{Brownian motion with drift}
We consider that $X=\{X_t\colon t\geq 0\}$ is a Brownian motion with drift,
\begin{equation}\label{eq:bm}
X_t=x+ct+\sigma W_t,\quad t\geq 0,
\end{equation}
where $\{W_t\colon t\geq 0\}$ is a standard Brownian motion.
The characteristic exponent of the process is 
$$
\Psi(z)={\sigma^2\over 2}z^2+cz,
$$ 
therefore, the denominator in the Wiener-Hopf factorization is zero when $\Psi(z)=r$ that has real roots $-r_I<0<r_S$, with
\begin{equation}\label{eq:roots_MB}
r_I=\frac{\sqrt{c^2+2r\sigma^2}+c}{\sigma^2}
\qquad\text{and}\qquad
r_S=
\frac{\sqrt{c^2+2r\sigma^2}-c}{\sigma^2}.
\end{equation}
As $r_Ir_S=2r/\sigma^2$, Wiener-Hopf factorization reads
$$
{r\over r-\Psi(z)}={2r\over\sigma^2(r_I+z)(r_S-z)}={r_I\over r_I+z}{r_S\over r_S-z},
$$
giving densities
\begin{align*}
f_I(x)&=r_Ie^{r_Ix},\quad x\leq 0,\\
f_S(x)&=r_Se^{-r_Sx},\quad x\geq 0,
\end{align*}
i.e. the formulas in \eqref{eq:inf_density} and \eqref{eq:sup_density} with $\pi_I=\pi_S=0$.
In conclusion, applying the results of Theorem \ref{teo:ejtriple}, we obtain the following results.

\begin{corollary}
The DG of Definition \ref{def:dg} for the Brownian motion with drift in \eqref{eq:bm} and  payoff functions $G_1(x)=x-\delta,$ $G_2(x)=x+\delta,$ has an optimal stopping rules  given by hitting times of the levels
$$
\aligned
x_I&=-\delta-\dfrac1{r_I}-\dfrac{\sqrt{c^2+2r\sigma^2}}{r}\dfrac{e^{r_Iu}-1}{e^{(r_I+r_S)u}-1},\\
x_S&=\delta+\dfrac1{r_S}+\dfrac{\sqrt{c^2+2r\sigma^2}}{r}\dfrac{e^{r_Su}-1}{e^{(r_I+r_S)u}-1},
\endaligned
$$
where $r_I,r_S$ are the roots in \eqref{eq:roots_MB},
and $u$ is the unique positive solution of the equation
\begin{equation*}
u-2\delta=\dfrac{\sqrt{c^2+2r\sigma^2}}{r}
\frac{(e^{r_Su}-1)(e^{r_Iu}-1)}{e^{(r_I+r_S)u}-1},
\end{equation*}
and value function \eqref{eq:value}
with
$$
A_I=-\frac1{r_I}\frac{e^{r_Su}-1}{e^{(r_I+r_S)u}-1},
\quad
 A_S=\frac1{r_S}\frac{e^{r_Iu}-1}{e^{(r_I+r_S)u}-1}.
$$
\vskip1mm\par\noindent
Moreover, as a consequence of the discussion in Subsection \ref{subsection:smooth}, in view of the fact that
the infimum and supremum of the process have no atoms, the value function is everywhere differentiable, i.e. we have smooth pasting.
\end{corollary}
\begin{remark} The symmetric case is obtained when $c=0$, with roots
$
-r_I=r_S={\sqrt{2r}/\sigma}.
$
Furthermore, in this symmetric case, $A_I=-A_S$, $-x_I=x_S$, and the function $V(x)$ in \eqref{eq:value} presents central symmetry around the origin.
\end{remark}

\subsubsection*{Numerical examples}

To illustrate our results in this case we consider two examples. 
In the first one we choose the standard Brownian motion, with $c=0$ and $\sigma=1$, and the second
example is the Brownian motion with drift the parameters $c=1$ and $\sigma=1$. We consider $r=1$ and $\delta=1$. 
The results obtained are
$x_I =  -x_S =  -1.6955$ 
and
$x_I =-1.2426$
and   
$x_S = 2.3659$ 
respectively.
In Figure \ref{v_bm}, we show the corresponding value functions.
\begin{figure}[ht]
\begin{center}
\begin{tabular}{cc}
\begin{tikzpicture}
\begin{axis}[
       axis lines=middle,
    axis line style=->,
    xmin = -4.5, xmax = 4.5,
    ymin = -4.5, ymax = 4.5,
    xtick distance = 2,
    ytick distance = 2,
    grid = both,
    minor tick num = 1,
    major grid style = {lightgray},
    minor grid style = {lightgray!25},
    width = 0.5\textwidth,
    height = 0.5\textwidth,
        enlargelimits,
    tick label style={font=\tiny}
]
    \addplot[
        domain = -4.5:-1.695513,
        samples = 200,
        smooth,
        thick,
        black]
        {(x)+1};
\addplot[
        domain = -1.695513:1.695513,
        samples = 200,
        smooth,
        thick,
        black,
    ]
{-0.0058*exp(-1.442*(x-1.695513))+0.0058*exp(1.442*(x+1.695513))};
\addplot[
        domain = 1.695513:4.5,
        samples = 200,
        smooth,
        thick,
        black,
    ] {(x)-1};    
     \addplot[
        domain = -4.5:4.5,
        samples = 200,
        thick,
        black,dashed]
        {(x)+1};
     \addplot[
        domain = -4.5:4.5,
        samples = 200,
        thick,
        black,dashed]
        {(x)-1};
        \addplot[
        domain = -4.5:-1.695513,
        samples = 200,
        smooth, thick,
        black,dotted]
        { x+1+0.7071-2*0.005797*exp(1.4142*(x+1.6955))};
        \addplot[
        domain = 1.695513:4.5,
        samples = 200,
        smooth, thick,
        black,dotted]
        {x-1-0.7071+2*0.005797*exp(-1.4142*(x-1.6955))};       
\end{axis}
\end{tikzpicture}
&
\begin{tikzpicture}
\begin{axis} [axis lines=middle,
    axis line style=->,
    xmin = -4.5, xmax = 4.5,
    ymin = -4.5, ymax = 4.5,
    xtick distance = 2,
    ytick distance = 2,
    grid = both,
    minor tick num = 1,
    major grid style = {lightgray},
    minor grid style = {lightgray!25},
    width = 0.5\textwidth,
    height = 0.5\textwidth,
        enlargelimits,
    tick label style={font=\tiny}
]
    \addplot[
        domain = -4.5:-1.2426,
        samples = 200,
        smooth,
        thick,
        black]
        {(x)+1};
\addplot[
        domain = -1.2426:2.3659,
      samples = 200,
     smooth,
    thick,
     black,
   ]
{-0.00002*exp(-2.732*(x-2.3659))+0.0973*exp(0.732*(x+1.2426))};
\addplot[domain = 2.3659:4.5,
        samples = 200,
        smooth,
        thick,
        black,
    ] {(x)-1};    
     \addplot[
        domain = -4.5:4.5,
        samples = 200,
        thick,
        black,dashed]
        {(x)+1};
     \addplot[
        domain = -4.5:4.5,
        samples = 200,
        thick,
        black,dashed]
        {(x)-1};
        \addplot[
        domain = -4.5:-1.2426,
        samples = 200,
        smooth, thick,
        black,dotted]
        { x+1+0.366-1.2679*0.0973*exp(0.7320*(x+1.2426))};
        \addplot[
        domain = 2.3659:4.5,
        samples = 200,
        smooth, thick,
        black,dotted]
        {x-1-1.3660+4.7320*0.00001777*exp(-2.7320*(x-2.3650))};       
\end{axis}
\end{tikzpicture}
\end{tabular}
\end{center}
\caption{ Brownian Motion.  Value function $V(x)$ (solid), $G_1(x)$, $G_2(x)$ (dashed) and $Q_I(x)$, $Q_S(x)$ (dotted). Left: symmetric case.  Right: asymmetric case.
}
\label{v_bm}
\end{figure}

\subsection{Cram\'er-Lundberg process}
We consider the Cram\'er-Lundberg process $X=\{X_t\colon t\geq 0\}$ with exponential jumps, given by 
\begin{equation}\label{eq:cramer}
X_t=x+ct-\sum_{i=1}^{N_t^{(1)}}Y_i^{(1)},
\end{equation}
where $N^{(1)}=\{N_t^{(1)}:t\geq 0\}$
is a  Poisson process with  intensity $\lambda_1$ and
$Y^{(1)}=\{Y_i^{(1)}\colon i\geq 1\}$ is a  sequence of independent identically distributed exponential random variables with parameter $\alpha_{1}$. 
The two processes $N^{(1)}$, and $Y^{(1)}$  are independent. 

 
The characteristic exponent of the process is 
$$
\Psi(z)= cz- \lambda_1 \frac{z}{\alpha_1+z},
$$ 
therefore, the denominator in the Wiener-Hopf factorization is zero when $\Psi(z)=r$ that has two roots $-r_I<0<r_S$, with 
$$
\aligned
r_I&={\sqrt{({c\alpha_1-\lambda_1-r})^2+4cr\alpha_1}+c\alpha_1-\lambda_1-r\over 2c},\\
r_S&={\sqrt{({c\alpha_1-\lambda_1-r})^2+4cr\alpha_1}-({c\alpha_1-\lambda_1-r})\over 2c}.
\endaligned
$$
The Wiener-Hopf factorization to determine the law of $S$ and $I$ is
$$
\frac{r}{r-\Psi(z)}= \left(\pi_I+ (1-\pi_I)\frac{r_I}{r_I+z}\right)
\left(\frac{r_S}{r_S-z}\right),$$
where $\pi_I=r/(r_S c)$, and we used that $r_Ir_S=r\alpha_1/c$.
We observe that the random variables $S$ has  exponential distributions
with parameter $r_S$ 
 and $-I$ has defective exponential distributions
with parameter $r_I$
and an atom at zero of  size $\pi_I$. 
The respective densities are
\begin{align*}
f_I(x)&=\pi_I\delta_0(x)+(1-\pi_I)r_Ie^{r_Ix},\quad x\leq 0,\\
f_S(x)&=r_Se^{-r_Sx},\quad x\geq 0,
\end{align*}
where $\delta_0(x)dx$ denotes the Dirac mass measure at $x=0$. Consequently, 
we obtain  \eqref{eq:inf_density} and \eqref{eq:sup_density}
with $\pi_I>0$ and $\pi_S=0.$
We then obtain the following result.
\begin{corollary}
The DG of Definition \ref{def:dg} for the Cram\'er Lundberg process in \eqref{eq:cramer} and 
payoff functions $G_1(x)=x-\delta,$ $G_2(x)=x+\delta,$  has an optimal stopping rules  given by hitting times of the levels
given by \eqref{eq:x1} and \eqref{eq:x2} with $\pi_S=0$,
where $u$ is the unique solution of  \eqref{eq:u}, and the value function is given in
\eqref{eq:value}.
\vskip1mm\par\noindent
Moreover, as a consequence of the discussion in Subsection \eqref{subsection:smooth}, 
in view of the fact that only the infimum has an atom,
we have smooth pasting at $x_S$ but the value function at $x_I$ is not differentiable.
The jump of the derivative can be computed by  \eqref{eq:atom}.
\end{corollary}
\subsubsection*{Numerical examples}

To illustrate our results in this process we consider  the Cram\'er-Lundberg  with $c=1$, $\lambda=1$ and $\alpha_1=1$.  We consider $r=1$ and $\delta=1$. 
The critical thresholds are
$x_I =-1.6127$
and 
$x_S  =1.4931$.

In Figure we show the corresponding value functions.

\begin{figure}[ht]
\label{lkf_v_cl}
\begin{center}
\begin{tikzpicture}
\begin{axis}[
    axis lines=middle,
    axis line style=->,
    xmin = -4.5, xmax = 4.5,
    ymin = -4.5, ymax = 4.5,
    xtick distance = 2,
    ytick distance = 2,
    grid = both,
    minor tick num = 1,
    major grid style = {lightgray},
    minor grid style = {lightgray!25},
    width = 0.5\textwidth,
    height = 0.5\textwidth,    enlargelimits,
    tick label style={font=\tiny}
]
    \addplot[
        domain = -4.5:-1.6127,
        samples = 200,
        smooth,
        thick,
        black]
        {(x)+1};
\addplot[
        domain = -1.6127:1.4931,
        samples = 200,
        smooth,
        thick,
        black,
    ] {-0.0904*exp(-0.618*(x-1.4931))+0.0038*exp(1.618*(x+1.6127))};
\addplot[
        domain = 1.4931:4.5,
        samples = 200,
        smooth,
        thick,
        black,
    ] {(x)-1};
     \addplot[
        domain = -4.5:4.5,
        samples = 200,
        smooth, thick,
        black,dashed]
        {(x)+1};
     \addplot[
        domain = -4.5:4.5,
        samples = 200,
        smooth, thick,
        black,dashed]
        {(x)-1};
        \addplot[
        domain = -4.5:-1.6127,
        samples = 200,
        smooth, thick,
        black,dotted]
        {x+1+0.6180-1.3819*0.003833*exp(1.6180*(x+1.6127))};
        \addplot[
        domain =  1.4931:4.5,
        samples = 200,
        smooth, thick,
        black,dotted]
        {x-1-0.6180+1.3819*0.0904*exp(-0.6180*(x-1.4931))};       
\end{axis}
\end{tikzpicture}
\end{center}
\caption
{Cram\'er-Lundberg process. Value function $V(x)$ (solid), $G_1(x)$, $G_2(x)$ (dashed) and $Q_I(x)$, $Q_S(x)$ (dotted).}
\end{figure}

\subsection{Compound Poisson process}
We consider the compound Poisson process $X=\{X_t\colon t\geq 0\}$ with double-sided exponential jumps, given by 
\begin{equation}\label{cpp}
X_t=x-\sum_{i=1}^{N_t^{(1)}}Y_i^{(1)}+\sum_{i=1}^{N_t^{(2)}}Y_i^{(2)},
\end{equation}
where $N^{(1)}=\{N_t^{(1)}\colon t\geq 0\}$
and $N^{(2)}=\{N_t^{(2)}\colon t\geq 0\}$ are two Poisson process 
 with respective positive intensities $\lambda_1,$ $\lambda_2$,
 the two sequences 
$Y^{(1)}=\{Y_i^{(1)}\colon i\geq 1\}$ and $Y^{(2)}=\{Y_i^{(2)}\colon i\geq 1\}$ are of independent identically distributed exponential random variables with respective positive parameters $\alpha_1$, $\alpha_2$. 
The four processes  $N^{(1)}$, $N^{(2)}$, $Y^{(1)}$, $Y^{(2)}$  are independent. 
 
The characteristic exponent of the process $X$ is given by 
$$
\Psi(z)=-\lambda_1\frac{z}{\alpha_1+z}+\lambda_2\frac{z}{\alpha_2-z},
$$ 
therefore the denominator in the Wiener-Hopf factorization $\Psi(z)=r$ that has two roots $-r_I$, $r_S$ are  the solutions of the equation
$$
(r+\lambda_1+\lambda_2)z^2+(\alpha_1(\lambda_2+r)-\alpha_2(\lambda_1+r))z-r\alpha_1\alpha_2=0,
$$
that satisfy
$$
-\alpha_1<-r_I<0<r_S<\alpha_2.
$$
The Wiener-Hopf factorization to determine the law of $S$ and $I$ is
\begin{align*}
{r\over r-\psi(z)}
&={r(z-\alpha_2)(z-\alpha_1)\over (r+\lambda+\mu)(z-r_S)(z-r_I)}
={r_Ir_S(\alpha_1+z)(\alpha_2-z)\over \alpha_1\alpha_2(r_I+z)(r_S-z)}\\
&=\left({r_I\over\alpha_1}+{\alpha_1-r_I\over\alpha_1}{r_I\over r_I+z}\right)
\left({r_S\over\alpha_2}+{\alpha_2-r_S\over\alpha_2}{r_S\over r_S-z}\right).
\end{align*}
In conclusion, due to the uniqueness of the factorization (see Thm. 5(ii) Ch. VI of \cite{Bertoin}),
we obtain
that the random variables 
$-I$  
and 
$S$
have (strictly) defective exponential distributions
with parameters $r_I$ and $r_S$,
and atoms at zero of respective sizes $\pi_I=r_I/\alpha_1$ and $\pi_S=r_S/\alpha_2$,
corresponding to  \eqref{eq:sup-fc} and \eqref{eq:inf-fc}, giving then the following result.

\begin{corollary}
The DG of Definition \ref{def:dg} for the compound Poisson process with double sided exponential jumps in \eqref{cpp} and payoff functions $G_1(x)=x-\delta,$ $G_2(x)=x+\delta,$  has optimal stopping rules  given by hitting times of the levels
 given by \eqref{eq:x1} and \eqref{eq:x2} with $\pi_I=r_I/\alpha_1$ and $\pi_S=r_S/\alpha_2$, where $u$ is the unique solution of  \eqref{eq:u}, and the value function is given in \eqref{eq:value}.
\vskip1mm\par\noindent
Moreover, as a consequence of the discussion in Subsection \eqref{subsection:smooth}, 
in view of the fact that both 
the infimum and the supremum have atoms,
 the value function at $x_S$ and $x_I$ is not differentiable.
The jump of the derivative at these points can be computed by  \eqref{eq:atom}.
\end{corollary}

\subsubsection*{Numerical examples}
To illustrate our results we consider two examples. In the first one we choose a symmetric case with $(\alpha_1 , \lambda_1 , \alpha_2 , \lambda_2 ) = (1, 1, 1, 1)$ and in the second example an asymmetric case
with $(\alpha_1 , \lambda_1 , \alpha_2 , \lambda_2 ) = (1,3, 3, 1)$. We consider $r=1$ and $\delta=1$. 
The results obtained are
$x_I= - 1.5901$, $x_S =1.5901$   and $x_I = - 3.7750$,   $x_S = 0.0834$ respectively.

In Figure \ref{v_cp} we show the corresponding value functions $V(x)$.
In both cases we observe that in $x_I$ and $x_S$ the value function is not differentiable.


\begin{figure}[ht]
\label{v_cp}
\begin{center}
\begin{tabular}{cc}
\begin{tikzpicture}
\begin{axis}[
    axis lines=middle,
    axis line style=->,
    xmin = -4.5, xmax = 4.5,
    ymin = -4.5, ymax = 4.5,
    xtick distance = 2,
    ytick distance = 2,
    grid = both,
    minor tick num = 1,
    major grid style = {lightgray},
    minor grid style = {lightgray!25},
    width = 0.5\textwidth,
    height = 0.5\textwidth,    enlargelimits,
    tick label style={font=\tiny}
]
    \addplot[
        domain = -4.5:-1.59,
        samples = 200,
        smooth,
        thick,
        black]
        {(x)+1};
\addplot[
        domain = -1.59:1.59,
        samples = 200,
        smooth,
        thick,
        black,
    ] {-0.112*exp(-0.577*(x-1.59))+0.112*exp(0.577*(x+1.59))};
\addplot[
        domain = 1.59:4.5,
        samples = 200,
        smooth,
        thick,
        black,
    ] {(x)-1};
     \addplot[
        domain = -4.5:4.5,
        samples = 200,
        smooth, thick,
        black, dashed]
        {(x)+1};
     \addplot[
        domain = -4.5:4.5,
        samples = 200,
        smooth, thick,
        black,dashed]
        {(x)-1};
        \addplot[
        domain =  -4.5:-1.59,
        samples = 200,
        smooth, thick,
        black,dotted]
        {x+1+0.7322-1.2679*0.1119*exp(0.5773*(x+1.5901))};
        \addplot[
        domain = 1.59:4.5,
        samples = 200,
        smooth, thick,
        black,dotted]
        {x-1-0.7322+1.2679*0.1119*exp(-0.5773*(x-1.5901))};
\end{axis}
\end{tikzpicture}
&
\begin{tikzpicture}
\begin{axis}[
       axis lines=middle,
    axis line style=->,
    xmin = -4.5, xmax = 4.5,
    ymin = -4.5, ymax = 4.5,
    xtick distance = 2,
    ytick distance = 2,
    grid = both,
    minor tick num = 1,
    major grid style = {lightgray},
    minor grid style = {lightgray!25},
    width = 0.5\textwidth,
    height = 0.5\textwidth,
    enlargelimits,
    tick label style={font=\tiny}]
    \addplot[
        domain = -4.5:-3.77,
        samples = 200,
        smooth,
        thick,
        black]
        {(x)+1};
\addplot[
        domain = -3.77:0.08,
        samples = 200,
        smooth,
        thick,
        black,
    ]
{-0.998*exp(-0.265*(x-0.08))+0.00001312*exp(2.265*(x+3.77))};
\addplot[
        domain = 0.0833:4.5,
        samples = 200,
        smooth,
        thick,
        black,
    ] {(x)-1};    
     \addplot[
        domain = -4.5:4.5,
        samples = 200,
        smooth, thick,
        black,dashed]
        {(x)+1};
     \addplot[
        domain = -4.5:4.5,
        samples = 200,
        smooth, thick,
        black,dashed]
        {(x)-1};
        \addplot[
        domain = -4.5:-3.77,
        samples = 200,
        smooth, thick,
        black,dotted]
        { x+1+2.7749-2.9250*0.0000131*exp(2.2649*(x+3.775))};
        \addplot[
        domain = 0.0833:4.5,
        samples = 200,
        smooth, thick,
        black,dotted]
        {x-1-0.1081+1.0263*0.9985*exp(-0.2649*(x-0.0834))};       
\end{axis}
\end{tikzpicture}
\end{tabular}
\end{center}
\caption{ Compound Poisson process.  Value function $V(x)$ (solid), $G_1(x)$, $G_2(x)$ (dashed) and $Q_I(x)$, $Q_S(x)$ (dotted). Left: symmetric case.  Right: asymmetric case.
}
\end{figure}


\section{Callable futures under the Kou model}
\label{section:callable}
We consider the mathematical model of a financial market introduced by Kou \cite{kou}. There are two assets, a riskless asset given by $B_t=B_0e^{rt},\ t\geq 0$,
and risky asset $S=\{S_t\colon t\geq 0\}$ given by
$$
S_t=e^{X_t}=S_0\exp\left(
at+\sigma W_t-\sum_{i=1}^{N_t^{(1)}}Y_i^{(1)}+\sum_{i=1}^{N_t^{(2)}}Y_i^{(2)}
\right),\quad S_0=e^x,
$$
where a L\'evy process  $X=\{X_t\colon t\geq 0\}$
consists in a drifted Brownian motion plus double-sided exponential jumps.
More precisely, 
$W=\{W_t\colon t\geq 0\}$ is a standard Wiener process, $a\in\R$ and $\sigma> 0$,
$N^{(1)}=\{N_t^{(1)}\colon t\geq 0\}$
and $N^{(2)}=\{N_t^{(2)}\colon t\geq 0\}$ are two Poisson process 
 with respective positive intensities $\lambda_1,$ $\lambda_2$,
 the two sequences 
$Y^{(1)}=\{Y_i^{(1)}\colon i\geq 1\}$ and $Y^{(2)}=\{Y_i^{(2)}\colon i\geq 1\}$ are of independent identically distributed exponential random variables with respective positive parameters $\alpha_1$, $\alpha_2$. 
The five processes  $W$, $N^{(1)}$, $N^{(2)}$, $Y^{(1)}$, $Y^{(2)}$  are independent. 


In this market model, a perpetual callable American  futures contract is introduced, 
defined in the following way. Consider two positive strikes $k<K$.
If the holder (here is the max player) chooses the execution time $\tau$, the contract is struck at $K$, he then receives $S_\tau-K$. The contract has the additional feature that the issuer (the min player), by paying a penalty
$K-k$, can also choose the execution time $\sigma$,
of the same contract struck at $k$ (i.e. is a callable contract). 
The issuer then pays $S_\sigma-k$ to the holder.
The resulting contract is a discounted Dynkin game for the process $X$ with
payoff functions
 $$
G_1(x)=e^x-K < G_2(x)=e^x-k.
$$


The characteristic exponent of the L\'evy process $X$ is
$$
\Psi(z)=az+\sigma^2{z^2\over 2}-\lambda_1\frac{z}{\alpha_1+z}+\lambda_2\frac{z}{\alpha_2-z},
$$ 
therefore, the denominator in the Wiener-Hopf factorization is zero when
$\psi(z)=r$, that has four real roots $-n_2<-\alpha_1<- n_1<0<p_1<\alpha_2<p_2$.
Moreover, the Wiener-Hopf factorization \eqref{eq:wh}
gives the densities
\begin{align}
\label{eq:densityI_cf}
f_I(x)&=\pi_I n_2e^{n_2x}+(1-\pi_I)n_1e^{n_1x},\quad &x\leq 0,\\
\label{eq:densityS_cf}
f_S(x)&=\pi_S p_1e^{-p_1x}+(1-\pi_S)p_2e^{-p_2x},\quad &x\geq 0.
\end{align}
It is important to note that the condition $p_1>1$ in necessary for the integrability condition \eqref{eq:ic} to hold.
The coefficients in \eqref{eq:densityI_cf} and \eqref{eq:densityS_cf} are
$$
\pi_I={n_1(\alpha_1-n_1)\over\alpha_1(n_2-n_1)}
\qquad\text{and}\qquad
\pi_S={p_2(\alpha_2-p_1)\over\alpha_2(p_2-p_1)}.
$$
(see (11) with $n=1$ in \cite{ruin}).


\subsection{Averaging functions and thresholds}

In order to find functions $Q_I(x)$ and $Q_S(x)$ s.t. \eqref{eq:equal} holds,
we compute
$\E_xQ_I(I)$ for $x\leq x_I$
and 
$\E_xQ_S(S)$ for $x\geq x_S$ 
(using  \eqref{eq:densityI_cf} and  \eqref{eq:densityS_cf}),
\begin{align}
\E_xQ_I(I)&=\int_{-\infty}^0Q_I(x+y)f_I(y)dy
=\int_{-\infty}^{x\wedge x_I}Q_I(z)f_I(z-x)dz\notag\\
&={\pi_I}e^{-n_2x}\int_{-\infty}^{x\wedge x_I} n_2e^{n_2z}Q_I(z)dz
+\left(1-{\pi_I}\right)e^{-n_1x}\int_{-\infty}^{x\wedge x_I}n_1e^{n_1z}Q_I(z)dz.
\label{eq:continuity1}
\end{align}
Analogously
\begin{align}
\E_xQ_S(S)&=\int_0^\infty Q_S(x+y)f_S(y)dy=\int_{x\vee x_S}^\infty Q_S(z)f_S(z-x)dz\notag\\
&={\pi_S}e^{p_1x}\int_{x\vee x_S}^\infty p_1e^{-p_1z}Q_S(z)dz
+\left(1-{\pi_S}\right)e^{p_2x}\int_{x\vee x_S}^\infty p_2e^{-p_2z}Q_S(z)dz.
\label{eq:continuity2}
\end{align}
For convenience, we introduce the notation
\begin{align*}
N_2&=\pi_Ie^{-{n_2}x_S}\int_{-\infty}^{x_I} n_2e^{n_2z}Q_I(z)\,dz,
&N_1&=(1-\pi_I)e^{-{n_1}x_S}\int_{-\infty}^{x_I} n_1e^{n_1z}Q_I(z)\,dz,\\
P_1&=\pi_Se^{p_1x_I}\int_{x_S}^\infty p_1e^{-p_1z}Q_S(z)\,dz,
&P_2&=(1-\pi_S)e^{p_2x_I}\int_{x_S}^\infty p_2e^{-p_2z}Q_S(z)\,dz.
\end{align*}
Two continuity equations are obtained evaluating \eqref{eq:equal} at the points $x_I$ and $x_S$, with the computations in \eqref{eq:continuity1} and \eqref{eq:continuity2}. 
With the above notation, letting $u=x_S-x_I$, we obtain
\begin{align}
e^{n_2u}N_2+e^{n_1{u}}N_1+P_1+P_2&=e^{x_I}-k,\tag{I}\label{eq:1}\\
N_2+N_1+e^{p_1u}P_1+e^{p_2u}P_2&=e^{x_S}-K.\tag{II}\label{eq:2}
\end{align}
In view of \eqref{eq:ve}, our guess for the value function for $x\in(x_I,x_S)$, is
\begin{multline}\label{eq:valuekou}
V(x)=\\
\begin{cases}
e^x-k, &\text{$x\leq x_I$},\\    
N_2e^{-n_2(x-x_S)}+N_1e^{-n_1(x-x_S)}+P_1e^{p_1(x-x_I)}+P_2e^{p_2(x-x_I)},&\text{ $x_I\leq x\leq x_S$},\\
e^x-K, &\text{$x\geq x_S$.}
\end{cases}
\end{multline}
As $\P(S=0)=\P(I=0)=0$, in view of Proposition \ref{proposition:sp},
we expect the smooth pasting property to hold, at both contact points.
We then impose the continuity of derivatives, by differentiation of \eqref{eq:valuekou} from the left and the right at the points $x_I$ and $x_S$. 
The resulting equations are
\begin{align}
-n_2e^{n_2u}N_2-n_1e^{-n_1u}N_1+p_1P_1+p_2P_2&=e^{x_I},\tag{III}\label{eq:3}\\
-n_2N_2-n_1N_1+p_1e^{p_1u}P_1+p_2e^{p_2u}P_2&=e^{x_S}\tag{IV}\label{eq:4}.
\end{align}


In view of \eqref{eq:equal} and our computations in \eqref{eq:continuity1}
and \eqref{eq:continuity2}, we conclude that each averaging functions
should be a linear combination of three exponentials
\begin{align*}    
Q_I(z)	&=B_1e^{p_1(z-x_I)}+B_2e^{p_2(z-x_I)}+B_3e^z+B_4,\\
Q_S(z)	&=C_1e^{-n_1(z-x_S)}+C_2e^{-n_2(z-x_S)}+C_3e^z+C_4.
\end{align*}
After plugging these formulas in \eqref{eq:equal} to find the coefficients, 
and imposing the continuity of the averaging functions at $x_I$ and $x_S$,
we arrive at 
\begin{align}
-{(n_2+p_1)(n_1+p_1)\over n_1n_2+p_1\bar n}P_1-{(n_2+p_2)(n_1+p_2)\over n_1n_2+p_2\bar n}P_2+\frac{(n_2+1)(n_1+1)}{n_1n_2+\bar n}e^{z_I}&=k,\tag{V}\label{eq:5}\\
-{(n_2+p_1)(n_2+p_2)\over p_1p_2+n_2\bar p}N_2-{(n_1+p_1)(n_1+p_2)\over p_1p_2+n_1\bar p}N_1+{(p_1-1)(p_2-1)\over p_1p_2-\bar p}e^{z_S}&=K,\tag{VI}\label{eq:6}
\end{align}
where
\begin{equation*}
    \bar p=\pi_Sp_1+(1-\pi_S)p_2 \qquad\text{and}\qquad \bar n =\pi_In_2+(1-{\pi_I})n_1.   
\end{equation*}
Equations \eqref{eq:1}-\eqref{eq:6} constitute a nonlinear system of six equations to find the six unknowns
$N_2,N_1,P_1,P_2,x_I,x_S$.
Once a solution is found, we should check $x_I<x_S$, and that $Q_I$ is negative, $Q_S$ positive. Condition \eqref{eq:equal} holds (because it was imposed to obtain the system), but condition \eqref{eq:geq} should also be checked.
\subsection*{Numerical example}
To illustrate our result, we consider parameters $a=0$, $\sigma^2=2$,  
$\alpha_1=\alpha_2=8/3$, $\lambda_1=\lambda_2=35/18$  and $r=9$.
Besides, we take the values $k=1$, $K=2$ in the perpetual callable futures contracts.
The Wiener-Hopf factorization is
\begin{align*}
{r\over r-\Psi(z)}
=\frac{64-9z^2}{z^4-20z^2+64}
=\left(
\frac12{2\over2+z}+\frac12{4\over4+z}
\right)\left(
\frac12{2\over2-z}+\frac12{4\over4-z}
\right),
\end{align*}
therefore
$n_2=p_2=4,\  n_1=p_1=2,\  \pi_I=\pi_S=1/2.$

To solve the system we implement a numerical scheme that, departing 
from $x_I^{(0)}=k$ and $x_S^{(0)}=K$ computes $N_2^{(0)},N_1^{(0)},P_1^{(0)},P_2^{(0)}$ by solving \eqref{eq:1}-\eqref{eq:4}
that is a linear system in the four unknowns. With these four values we compute
$x_I^{(1)}$ and $x_S^{(1)}$ using \eqref{eq:5} and \eqref{eq:6}.
Afterwards we compute  $,N_2^{(1)},N_1^{(1)},P_1^{(1)},P_2^{(1)}$ and so on.
This scheme quickly converges to a solution such that the value function in \eqref{eq:valuekou} verifies all the requirements of Theorem \ref{theorem:1},
see Fig. \ref{lkf_v_cl}.
The solution is
$x_I=-0.211$ and
$x_S= 1.195$, with
$(N_2,N_1,P_1,P_2)= (-\num{3.347e-4},\ -\num{9.473e-3},\ \num{5.928e-2}, \ \num{1.181e-3}).$
\begin{figure}
\label{lkf_v_cl}
\begin{center}
\begin{tikzpicture}
\begin{axis}[
    axis lines=middle,
    axis line style=->,
    xmin = -2.0, xmax = 2.0,
    ymin = -2.0, ymax = 2.0,
    xtick distance = 2,
    ytick distance = 2,
    grid = both,
    minor tick num = 1,
    major grid style = {lightgray},
    minor grid style = {lightgray!25},
    width = 0.7\textwidth,
    height = 0.7\textwidth,    enlargelimits,
    tick label style={font=\tiny}
]
    \addplot[
        domain = -2.0:-0.2107,
        samples = 200,
        smooth,
        thick,
        black]
        {g2(x)};
\addplot[
        domain = -0.2107:1.1954,
        samples = 200,
        smooth,
        thick,
        black,
    ] {v(x))};
\addplot[
        domain = 1.1954:2.0,
        samples = 200,
        smooth,
        thick,
        black,
    ] {(g1(x)};
     \addplot[
        domain = -2.0:2.0,
        samples = 200,
        smooth, thick,
        black,dashed]
        {g2(x)};
     \addplot[
        domain = -2.0:2.0,
        samples = 200,
        smooth, thick,
        black,dashed]
        {g1(x)};
        \addplot[
        domain = -2.0:-0.2107,
        samples = 200,
        smooth, thick,
        black,dotted]
        {qi(x)};
        \addplot[
        domain =  1.1954:2.0,
        samples = 200,
        smooth, thick,
        black,dotted]
        {qs(x)};       
\end{axis}
\end{tikzpicture}
\end{center}
\caption
{Callable future under the Kou model. Value function $V(x)$ (solid), $G_1(x)$, $G_2(x)$ (dashed) and $Q_I(x)$, $Q_S(x)$ (dotted).}
\end{figure}
\section{Conclusions} 
\label{sec:conclusion}
We present a verification theorem to solve a Dynkin game
driven by a L\'evy processes.
The theorem requires finding two averaging functions,
that give an explicit form of the value function
in terms of the infimum and the supremum of the process.
The optimal stopping times result to be the entry times of the 
support sets of these averaging functions.
The main challenge in this problem is managing the 
characteristic
overshoot caused by the jumps of the L\'evy processes, 
rendering traditional solution techniques, used for diffusions, ineffective.
Notably, in some instances, 
when Wiener-Hopf factors can be computed explicitly, 
we are able to identify the averaging functions and solve the problem completely. 
We first consider Brownian motion with drift, the Cram\'er-Lundberg process and the compound Poisson process with affine payoffs.
Secondly, we present a financial application, a 
perpetual callable American futures  
in the Kou's market model.
In these examples we discuss the smooth pasting property,
that does not always hold.

\bibliographystyle{apalike}
\bibliography{bibliografia_articulo}

\begin{thebibliography}{}

\bibitem[Alario-Nazaret et~al., 1982]{alarionazaret}
Alario-Nazaret, M., Lepeltier, J.~P., and Marchal, B. (1982).
\newblock Dynkin games.
\newblock In Kohlmann, M. and Christopeit, N., editors, {\em Stochastic
  Differential Systems}, pages 23--32, Berlin, Heidelberg. Springer Berlin
  Heidelberg.

\bibitem[Alili and Kyprianou, 2005]{alili}
Alili, L. and Kyprianou, A. (2005).
\newblock {Some remarks on first passage of L\'evy processes, the American put
  and pasting principles}.
\newblock {\em Annals of Applied Probability}, 15(3):2062--2080.

\bibitem[Alvarez, 2008]{alvarez}
Alvarez, L. H.~R. (2008).
\newblock {A class of solvable stopping games}.
\newblock {\em Appl. Math. Optim.}, 58(3):291--314.

\bibitem[Asmussen et~al., 2004]{asmussen}
Asmussen, S., Avram, F., and Pistorius, M. (2004).
\newblock {Russian and American put options under exponential phase-type L\'evy
  motion}.
\newblock {\em Stoch. Proc. Appl.}, 109:79--111.

\bibitem[Baurdoux and Kyprianou, 2008]{bardoux}
Baurdoux, E. and Kyprianou, A. (2008).
\newblock {The McKean stochastic game driven by a spectrally negative Lévy
  process}.
\newblock {\em Electronic Journal of Probability}, 13(none):173 -- 197.

\bibitem[Baurdoux and Kyprianou, 2009]{bardoux2}
Baurdoux, E.~J. and Kyprianou, A.~E. (2009).
\newblock {The Shepp-Shiryaev stochastic game driven by a spectrally negative
  L\'evy process}.
\newblock {\em Theory of probability and its applications}, 53(3):481--499.

\bibitem[Bertoin, 1996]{Bertoin}
Bertoin, J. (1996).
\newblock {\em L{\'e}vy Processes}.
\newblock Cambridge Tracts in Mathematics. Cambridge University Press.

\bibitem[Chan, 2005]{chan}
Chan, T. (2005).
\newblock {\em {Pricing Perpetual American options driven by spectrally
  one-sided levy processes}}.

\bibitem[Christensen et~al., 2013]{CS}
Christensen, S., Salminen, P., and Ta, B.~Q. (2013).
\newblock Optimal stopping of strong markov processes.
\newblock {\em Stochastic Processes and their Applications}, 123(3):1138--1159.

\bibitem[Christensen and Schultz, 2024]{christensenschultz}
Christensen, S. and Schultz, B. (2024).
\newblock On the existence of markovian randomized equilibria in dynkin games
  of war-of-atrition-type.
\newblock {\em arXiv:2406.09820v2}.

\bibitem[Cvitanić and Karatzas, 1996]{cvitanickaratzas}
Cvitanić, J. and Karatzas, I. (1996).
\newblock Backward stochastic differential equations with reflection and dynkin
  games.
\newblock {\em The Annals of Probability}, 24(4):2024--2056.

\bibitem[Darling et~al., 1972]{darling}
Darling, D.~A., Liggett, T., and Taylor, H.~M. (1972).
\newblock {Optimal stopping for partial sums}.
\newblock {\em Ann. Math. Statist.}, 43:1363--1368.

\bibitem[Dynkin, 1969]{dynkin1969}
Dynkin, E.~B. (1969).
\newblock {A game-theoretic version of an optimal stopping problem}.
\newblock {\em Dokl. Akad. Nauk SSSR}, 185(1):16--19.

\bibitem[Ekstr{\"o}m and Peskir, 2008]{ekstrom2008}
Ekstr{\"o}m, E. and Peskir, G. (2008).
\newblock {Optimal Stopping Games for Markov Processes}.
\newblock {\em SIAM J. Control. Optim.}, 47:684--702.

\bibitem[Ekström and Villeneuve, 2006]{ekstromvilleneuve}
Ekström, E. and Villeneuve, S. (2006).
\newblock On the value of optimal stopping games.
\newblock {\em The Annals of Applied Probability}, 16(3):1576--1596.

\bibitem[Emmerling, 2012]{emmerling_2012}
Emmerling, T.~J. (2012).
\newblock Perpetual cancellable american call option.
\newblock {\em Mathematical Finance}, 22(4):645--666.

\bibitem[Frid, 1969]{frid1969}
Frid, E.~B. (1969).
\newblock {The Optimal Stopping Rule for a Two-Person Markov Chain with
  Opposing Interests}.
\newblock {\em Theory of Probability \& Its Applications}, 14(4):713--716.

\bibitem[Gapeev and Kühn, 2005]{gapeev}
Gapeev, P.~V. and Kühn, C. (2005).
\newblock {Perpetual convertible bonds in jump-diffusion models}.
\newblock {\em Statistics \& Risk Modeling}, 23(1):15--31.

\bibitem[Gusein-Zade, 1969]{gusein1969}
Gusein-Zade, S.~M. (1969).
\newblock {On a Game Connected with the Wiener Process}.
\newblock {\em Theory of Probability \& Its Applications}, 14(4):701--704.

\bibitem[Hubalek and Kyprianou, 2010]{hk2010}
Hubalek, F. and Kyprianou, E. (2010).
\newblock {Old and New Examples of Scale Functions for Spectrally Negative
  L\'evy Processes}.
\newblock {\em Progress in Probability}, 63(none):119--145.

\bibitem[Jacod and Shiryaev, 1987]{js}
Jacod, J. and Shiryaev, A.~N. (1987).
\newblock {\em {Limit theorems for stochastic processes}}, volume 288.
\newblock Springer-Verlag Berlin.

\bibitem[Kifer, 2000]{kifer_2000}
Kifer, Y. (2000).
\newblock Game options.
\newblock {\em {Finance and Stochastics}}, 4:443--463.

\bibitem[Kou, 2002]{kou}
Kou, S. (2002).
\newblock {A jump diffusion model for option pricing.}
\newblock {\em Management Science}, 48:1086--1101.

\bibitem[Kuznetsov, 2010a]{kuznetsov:2010}
Kuznetsov, A. (2010a).
\newblock {Wiener-Hopf factorization for a family of Lévy processes related to
  theta functions}.
\newblock {\em Journal of Applied Probability}, 47(4):1023--1033.

\bibitem[Kuznetsov, 2010b]{kuznetsov}
Kuznetsov, A. (2010b).
\newblock {Wiener–Hopf factorization and distribution of extrema for a family
  of Lévy processes}.
\newblock {\em The Annals of Applied Probability}, 20(5):1801 -- 1830.

\bibitem[Kuznetsov et~al., 2013]{kkr}
Kuznetsov, A., Kyprianou, A.~E., and Rivero, V. (2013).
\newblock {\em {The Theory of Scale Functions for Spectrally Negative L{\'e}vy
  Processes}}, pages 97--186.
\newblock Springer Berlin Heidelberg, Berlin, Heidelberg.

\bibitem[Kyprianou, 2004]{kyprianou_2004}
Kyprianou, A.~E. (2004).
\newblock {Some calculations for Israeli options}.
\newblock {\em Finance and Stochastics}, 8:73--86.

\bibitem[Kyprianou, 2006]{Kyprianou}
Kyprianou, A.~E. (2006).
\newblock {\em {Introductory Lectures on Fluctuations of L{\'e}vy Processes
  with Applications}}.
\newblock Springer.

\bibitem[Lewis and Mordecki, 2008]{lm}
Lewis, A.~L. and Mordecki, E. (2008).
\newblock {Wiener-Hopf Factorization for Lévy Processes Having Positive Jumps
  with Rational Transforms}.
\newblock {\em Journal of Applied Probability}, 45(1):118--134.

\bibitem[Mordecki, 2002]{mordecki}
Mordecki, E. (2002).
\newblock {Optimal stopping and perpetual options for L\'evy processes.}
\newblock {\em Finance Stoch.}, 6(4):473--493.

\bibitem[Mordecki, 2003]{ruin}
Mordecki, E. (2003).
\newblock Ruin probabilities for {L}\'evy processes with mixed-exponential
  negative jumps.
\newblock {\em Teor. Veroyatnost. i Primenen.}, 48(1):188--194.

\bibitem[Mordecki and Oli{\'u}, 2021]{Oliu}
Mordecki, E. and Oli{\'u}, F. (2021).
\newblock {Two-sided optimal stopping for Lévy processes}.
\newblock {\em Electronic Communications in Probability}, 26(none):1 -- 12.

\bibitem[Mordecki and Salminen, 2007]{mordecki-salminen}
Mordecki, E. and Salminen, P. (2007).
\newblock {Optimal stopping of Hunt and Lévy processes}.
\newblock {\em Stochastics}, 79(3-4):233--251.

\bibitem[Neveu, 1975]{neveu}
Neveu, J. (1975).
\newblock {\em {Discrete-parameter martingales}}, volume Vol. 10 of {\em
  North-Holland Mathematical Library}.
\newblock North-Holland Publishing Co., Amsterdam-Oxford; American Elsevier
  Publishing Co., Inc., New York, revised edition.
\newblock Translated from the French by T. P. Speed.

\bibitem[Peskir, 2008]{peskir2008}
Peskir, G. (2008).
\newblock {Optimal Stopping Games and Nash Equilibrium}.
\newblock {\em Theory of Probability \& Its Applications}, 53(3):558--571.

\bibitem[Peskir and Shiryaev, 2006]{peskir2006}
Peskir, G. and Shiryaev, A.~N. (2006).
\newblock {\em {Optimal Stopping and Free-Boundary Problems}}.
\newblock Birk\"auser.

\bibitem[Rogozin, 1966]{rogozin}
Rogozin, B. (1966).
\newblock {On distributions of functionals related to boundary problems for
  processes with independent increments}.
\newblock {\em Theor. Probab. Appl.}, 11:580--591.

\bibitem[Sato, 2013]{sato}
Sato, K.-i. (2013).
\newblock {\em L\'evy processes and infinitely divisible distributions},
  volume~68 of {\em Cambridge Studies in Advanced Mathematics}.
\newblock Cambridge University Press, Cambridge, revised edition.
\newblock Translated from the 1990 Japanese original.

\bibitem[Snell, 1952]{snell}
Snell, J.~L. (1952).
\newblock {Applications of martingale system theorems}.
\newblock {\em Transactions of the American Mathematical Society}, pages
  293--312.

\bibitem[Touzi and Vieille, 2002]{touzivieille}
Touzi, N. and Vieille, N. (2002).
\newblock Continuous-time dynkin games with mixed strategies.
\newblock {\em SIAM Journal on Control and Optimization}, 41(4):1073--1088.

\end{thebibliography}

\end{document}